\documentclass[final]{siamltex}

\usepackage{algorithm2e}
\usepackage{amsmath,amssymb,amsfonts,mathdots,enumerate}
\usepackage{enumerate,color,graphicx,setspace}
\usepackage{listings}
\usepackage{hyperref}
\newcommand{\eps}{\varepsilon}
\newcommand{\To}{\rightarrow}
\newcommand{\N}{{\mathbb{N}}}
\newcommand{\R}{{\mathbb{R}}}
\newcommand{\C}{{\mathbb{C}}}
\newcommand{\F}{{\mathbb{F}}}

\newcommand{\I}{{\mathbb{I}}}
\newcommand{\abs}[1]{\left\vert #1\right\vert}
\newcommand{\norm}[1]{\left\Vert #1\right\Vert}

\DeclareMathOperator{\RE}{Re}

\DeclareMathOperator{\im}{im}
\DeclareMathOperator{\Gl}{Gl}
\DeclareMathOperator{\spa}{span}

\DeclareMathOperator{\normalrank}{normalrank}

\DeclareMathOperator{\condest}{condest}

\newcommand {\cV}       {{\cal V}}

\title{The SDA Method for Numerical Solution of Lur'e Equations}

\author{Federico Poloni\footnotemark[1]\and Timo Reis\footnotemark[2]}

\begin{document}

\maketitle
\renewcommand{\thefootnote}{\fnsymbol{footnote}}
\footnotetext[1]{Scuola Normale Superiore, piazza dei Cavalieri 7,
56126 Pisa, Italy ({\tt f.poloni@sns.it}).}
\footnotetext[2]{Institut f\"ur Numerische Simulation, Technische
Universit\"at Hamburg-Harburg, Schwarzenbergstra{\ss}e 95 E, 21073 Hamburg,
Germany ({\tt timo.reis@tu-harburg.de}). Supported by the DFG Research Center
``{\sc Matheon} - Mathematics for Key Technologies'' in Berlin.}
\renewcommand{\thefootnote}{\arabic{footnote}}

\begin{abstract}
We introduce a~numerical method for the numerical solution of the so-called Lur'e matrix equations that arise in balancing-related model reduction and linear-quadratic infinite time horizon optimal control.
Based on the fact that the set of solutions can be characterized in terms of deflating subspaces of even matrix pencils, an~iterative scheme is derived that converges linearly to the maximal solution.
\end{abstract}


%\begin{AMS}
%\end{AMS}

\pagestyle{myheadings}
\thispagestyle{plain}
\markboth{F. POLONI AND T. REIS}{LUR'E EQUATIONS}

\section{Introduction}
For given matrices $A,Q\in\C^{n,n}$ with $Q=Q^*$ and $B,C\in\C^{n,m}$, $R\in\C^{m,m}$, we consider {\em Lur'e equations}
\begin{equation}
\begin{aligned}
A^*X+XA+Q&=K^*K,\\
XB+C&=K^*L,\\
R&=L^*L,
\end{aligned}\label{eq:lure}
\end{equation}
that have to be solved for $(X,K,L)\in\C^{n,n}\times\C^{p,n}\times\C^{p,m}$ with $X=X^*$ and $p$ as small as possible. Equations of type (\ref{eq:lure}) were first introduced by {\sc A.I.~Lur'e} \cite{Lur51} in 1951 (see \cite{Bar07a} for an historical overview) and play a fundamental role in systems theory, e.g.\ since properties like dissipativity of linear systems can be characterized via their solvability \cite{Ande66,AndNew68,AndeV73,Wil72b}. This type of equations moreover appears in the infinite time horizon linear-quadratic optimal control problem \cite{CleAnd77a,CleAnd77b,CleAnd78,Yak85,ZhoDoyGlo96}, spectral factorization \cite{CleGlo89,CALM97} as well as in balancing-related model reduction \cite{ChenWen95,GugeA04,OpdJon88,PhilDS03,ReiSty08}. In the case where $R$ is invertible, the matrices $K$ and $L$ can be eliminated by obtaining the algebraic Riccati equation
\begin{equation}
A^*X+XA-(XB+C)R^{-1}(XB+C)^*+Q=0.
\label{eq:are}
\end{equation}
Whereas this type is well-explored both from an~analytical and numerical point of view \cite{Lanc95,Wil71,Ben97}, the case of singular $R$ has received comparatively little attention. However, the singularity of $R$ is often a~structural property of the system to be analyzed \cite{ReiSty08b} and can therefore not be avoided by arguments of genericity.

Two approaches exist for the numerical solution of Lur'e equations with (possibly) singular $R$. 
The works \cite{JacSp71b,WeWaSp94} present an~iterative technique for the elimination of variables corresponding to $\ker R$. After a finite number of steps this leads to a Riccati equation. This also gives an equivalent solvability criterion that is obtained by the feasibility of this iteration.
The most common approach to the solution of Lur'e equations is the slight perturbation of $R$ by $\eps I_m$ for some $\eps>0$. Then by using the invertibility of $R+\eps I$, the corresponding perturbed Lur'e equations are now equivalent to the Riccati equation
\begin{equation}
A^*X_\eps+X_\eps A-(XB+C)(R+\eps I)^{-1}(X_\eps B+C)^*+Q=0.
\label{eq:are_pert}
\end{equation}
It is shown in \cite{JacSp71a,Tre87} that certain corresponding solutions $X_\eps$ then converge to a solution of (\ref{eq:lure}).\\
Whereas the first approach has the great disadvantage that it relies on successive nullspace computations (which may be arbitrarily an ill-conditioned numerical problem), the big problem of the perturbation approach is that there exist no estimates for the perturbation error $\|X-X_\eps\|$ and, furthermore, the numerical condition of the Riccati equation (\ref{eq:are_pert}) increases drastically as $\eps$ tends to $0$.\\
From a~theoretical point of view, Lur'e equations have been investigated in \cite{CALM97,Bar07a}. The solution set is completely characterized in \cite{Rei09} via the consideration of the matrix pencil
\begin{equation}
s\mathcal{E}-\mathcal{A}=\begin{bmatrix}0&-sI+A&B\\sI+A^*&Q&C\\B^*&C^*&R\end{bmatrix}.\label{eq:evmatpen}
\end{equation}
This pencil has the special property of being {\em even}, that is $\mathcal{E}$ is skew-Hermitian and $\mathcal{A}$ is Hermitian.
Solvability of (\ref{eq:lure}) is characterized via the eigenstructure of this pencil. It is furthermore shown that there exists some correspondence to {\em deflating subspaces} of (\ref{eq:evmatpen}). That is, a generalization of the concept of invariant subspaces to matrix pencils\cite{Gant59}. Under some slight additional conditions of the pair $(A,B)$, it is shown in \cite{Rei09} that there exists a~so-called {\em maximal solution} $X$. In this case, maximal means that $X$ is, in terms of semi-definiteness, above all other solutions of the Lur'e equations. This solution is of particular interest in optimal control as well as in model reduction.

Based on these theoretical results of \cite{Rei09}, we will set up an iterative scheme that converges linearly to the maximal solution.

The paper is organized as follows. Section~\ref{sec:prelim} introduces the notation and contains
some required control and matrix theoretic background, in particular a normal form for even matrix pencils is introduced.
In Section \ref{sec:solv} we briefly present some results from \cite{Rei09} which connect the spectral properties of the even pencil
(\ref{eq:evmatpen}) to the solvability and the solutions of the Lur'e equations. An outline of the structured doubling algorithm for the solution of control problems is presented in Section~\ref{sec:SDA}. Section~\ref{sec:ssfilemma} contains a method for transforming a matrix pencil in the form required by the SDA. In Section~\ref{sec:reduction}, we apply this method to the singular control problem to obtain a reduced discrete-time pencil associated to a Lur'e equation. Section~\ref{sec:details} deals with the details of implementing SDA for this problem, and in particular with the choice of the parameter $\gamma$ in the Cayley transform. We present several numerical experiments to illustrate the benefits of this approach in Section~\ref{sec:experiments}, and presents our conclusions in Section~\ref{sec:conc}.

\section{Control and Matrix Theoretic Preliminaries}\label{sec:prelim}
Throughout the paper real and complex numbers are denoted by $\R$ and $\C$, the open left and right half-planes by $\C^-$ and $\C^+$,
respectively. The symbol $i$ stands for the imaginary unit, $i\R$ denotes the imaginary axis and by $\overline{z}$ we denote the conjugate transpose of $z\in\C$.
Natural numbers excluding and including $0$ are denoted by $\N$ and $\N_0$, respectively. The spaces of $n\times m$ complex matrices are denoted by
$\C^{n,m}$, and the set of invertible and complex $n\times n$ matrices by $\Gl_n(\C)$.
The matrices $A^T$ and $A^*$ denote, respectively, the transpose and the conjugate transpose of
$A\in\C^{n,m}$, and $A^{-T}=(A^{-1})^T$, $A^{-*}=(A^{-1})^*$. We denote by
$\rank(A)$ the rank, by $\im A$ the image, by $\ker A$ the
kernel, by $\sigma(A)$ the spectrum of a~matrix $A$.
For Hermitian matrices $P,Q\in\C^{n,n}$, we write $P>Q$
($P\geq Q$) if $P-Q$ is positive (semi-)definite.\\
For a rational matrix-valued function $\Phi:\C\backslash D\To\C^{n,m}$, where $D\subset\C$ is the finite set of poles, we define the normal rank by
$\normalrank\Phi=\max_{s\in\C\backslash D}\rank\Phi(s)$.\\
With
$A_i\in\C^{n_i,m_i}$ with $m_i,n_i\in\N_0$ for $i=1,\ldots,k$, we denote the block diagonal matrix by %with block diagonal matrix elements $A_1,\ldots,A_k$ by
$\diag(A_1,\ldots, A_k)$.
An~identity matrix of
order $n$ is denoted by $I_n$ or simply by $I$. The zero $n\times
m$ ($n\times n$) matrix is denoted by $0_{n,m}$ (resp.\ $0_{n}$) or simply by $0$. Moreover, for $k\in\N$ we introduce the following special matrices $J_k,M_k,N_k\in\R^{k,k}$, $K_k,L_k\in\R^{k-1,k}$ with
\[
\begin{aligned}
J_k&=\begin{bmatrix}&&1\\&\iddots&\\1&&\end{bmatrix},&
K_k&=\begin{bmatrix}0&1&\\&\ddots&\ddots&\\&&0&1\end{bmatrix},&
L_k&=\begin{bmatrix}1&0&\\&\ddots&\ddots&\\&&1&0\end{bmatrix},\\
M_k&=\begin{bmatrix}&&1&0\\&\iddots&\iddots&\\1&\iddots&&\\0&&&\end{bmatrix},&
N_k&=\begin{bmatrix}0&1&\\&\ddots&\ddots&\\&&\ddots&1\\&&&0\end{bmatrix}.
\end{aligned}
\]
\begin{definition}
Let $s E-A$ be a matrix pencil with $E,A\in\R^{m,n}$. Then
 $s E-A$ is called \emph{regular} if $m=n$ and
$\normalrank (s E-A)=n$.
%If $s E-A$ is not regular, then it is said to be \emph{singular}.

A pencil $s E-A$ is called \emph{even} if
$E=-E^*$ and $A=A^*$. A pencil with $E,A \in \R^{2n,2n}$ is called \emph{symplectic} if $EJE^T=AJA^T$, where
\[
J=\begin{bmatrix}0 & I_n\\-I_n & 0\end{bmatrix}.
\]
\end{definition}
Many properties of a matrix pencil can be
characterized in terms of the \emph{Kronecker canonical form (KCF)}.

%\begin{minipage}[c]{8cm}
\renewcommand{\arraystretch}{1.5}
\begin{table}[!h]
\begin{center}
\begin{tabular}{|l|l|l|l|}
\hline Type&Size&$\mathcal{C}_j(s)$&Parameters\\
\hline\hline
 W1&$k_j\times k_j$ & $(s-\lambda) I_{k_j}-N_{k_j}$ & $k_j\in\N$, $\lambda\in\C$ \\\hline
 W2&$k_j\times k_j$ & $sN_{k_j}-I_{k_j}$ & $k_j\in\N$  \\\hline
 W3&$(k_j-1)\times k_j$ & $sK_{k_j}-L_{k_j}$ & $k_j\in\N$  \\\hline
 W4&$k_j\times (k_j-1)$ & $sK^T_{k_j}-L^T_{k_j}$ & $k_j\in\N$  \\\hline
\end{tabular}~\\~\\\caption{Block types in Kronecker canonical form}\label{tab:wcf}\end{center}
\end{table}
%\end{minipage}

\begin{theorem}\cite{Gant59}\label{thm:wcf}
For a matrix pencil $sE-A$ with $E,A\in\C^{n,m}$, there exist matrices
$U_l\in\Gl_n(\C)$, $U_r\in\Gl_m(\C)$, such that
\begin{equation}U_l(s E-A)U_r=\diag(\mathcal{C}_1(s),\ldots,\mathcal{C}_k(s)), \label{eq:kronform}
\end{equation}
where each of the pencils $\mathcal{C}_j(s)$ is of one of the types presented in Table \ref{tab:wcf}.
~\\
The numbers $\lambda$ appearing in the blocks of type W1 are called the {\em (generalized) eigenvalues} of $sE-A$. Blocks of type W2 are said to be corresponding to infinite eigenvalues.\\
\end{theorem}
~\\
A~special modification of the KCF for even matrix pencils, the so-called {\em even Kronecker canonical form (EKCF)} is presented in\cite{Tho76}.
\renewcommand{\arraystretch}{1.5}
\begin{table}[h!]\begin{center}
\begin{tabular}{|l|l|l|l|}
\hline Type&Size&$\mathcal{D}_j(s)$&Parameters\\
\hline\hline&&&\\[-0.4cm]
 E1&$2k_j\times 2k_j$ & {\renewcommand{\arraystretch}{1}$\begin{bmatrix}0_{k_j,k_j}&(\lambda\!-\!s)I_{k_j}\!\!-\!N_{k_j}\\(\overline{\lambda}\!+\!s)I_{k_j}\!\!-\!N_{k_j}^T&0_{k_j,k_j}\end{bmatrix}$} &\renewcommand{\arraystretch}{1.5}$k_j\in\N$, $\lambda\in\C^+$ \\&&&\\[-0.5cm]\hline
 E2&$k_j\times k_j$&$\epsilon_j((-is-\mu)J_{k_j}+M_{k_j})$ &{\renewcommand{\arraystretch}{1}$\!\!\!\!\begin{array}{l}\\[-0.3cm]k_j\in\N, \mu\in\R,\\\epsilon_j\in\{-1,1\}\\[-0.35cm] \phantom{x}\end{array}$} \\\hline
 E3&$k_j\times k_j$ & $\epsilon_j (isM_{k_j}+J_{k_j})$ &{\renewcommand{\arraystretch}{1}$\!\!\!\!\begin{array}{l}\\[-0.3cm]k_j\in\N,\\\epsilon_j\in\{-1,1\}\\[-0.35cm] \phantom{x}\end{array}$}\\\hline&&&\\[-0.4cm]
 E4&{\renewcommand{\arraystretch}{1}$\!\!\!\begin{array}{l}(2k_j\!-\!1)\times\\(2k_j\!-\!1)\end{array}$} & {\renewcommand{\arraystretch}{1}$\begin{bmatrix}0_{k_j-1,k_j-1}&-sK_{k_j}+L_{k_j}\\sK_{k_j}^T+L_{k_j}^T&0_{k_j,k_j}\end{bmatrix}$}
&$k_j\in\N$\renewcommand{\arraystretch}{1.5}\\[0.4cm]\hline
\end{tabular}~\\~\\\caption{Block types in even Kronecker canonical form}\label{tab:ewcf}\end{center}
\end{table}
\renewcommand{\arraystretch}{1}

\begin{theorem}\cite{Tho76}\label{thm:ewcf}
For an even matrix pencil $sE-A$ with $E,A\in\C^{n,n}$, there exists a matrix
$U\in\Gl_n(\C)$ such that
\begin{equation}U^*(s E-A)U=\diag(\mathcal{D}_1(s),\ldots,\mathcal{D}_k(s)), \label{eq:kronform_ev}
\end{equation}
where each of the pencils $\mathcal{D}_j(s)$ is of one of the types presented in Table~\ref{tab:ewcf}.
~\\
The numbers $\eps_j$ in the blocks of type E2 and E3 are called the {\em block signatures}.\\
\end{theorem}
The blocks of type E1 contain pairs pairs $(\lambda,-\overline{\lambda})$ of generalized eigenvalues. The blocks of type E2 and E3 respectively correspond to the purely imaginary and infinite eigenvalues. Blocks of type E4 consist of a combination of blocks that are equivalent to those of type W3 and W4.\\

\begin{definition}\label{alg_struct}
A subspace $\cV\subset\C^N$ is called \emph{(right) deflating subspace}
for the pencil $sE-A$ with $E,A\in\C^{M,N}$ if for a matrix $V\in\C^{N,k}$ with
full column rank and $\im V=\cV$,
there exists an $l\leq k$ and matrices $W\in\C^{M,l}$, $\widetilde{E},\widetilde{A}\in\C^{l,k}$ with
\begin{equation}
(sE-A)V=W(s\widetilde{E}-\widetilde{A}),\label{eq:defl_pencil_form}
\end{equation}
\end{definition}
\begin{definition}
An eigenvalue $\lambda$ of a matrix pencil is called \emph{c-stable}, \emph{c-critical} or \emph{c-unstable} respectively if $\RE(\lambda)$ is smaller than, equal to, or greater than 0. A right deflating subspace is called \emph{c-stable} (resp. \emph{c-unstable}) if it contains only c-stable (resp. c-unstable) eigenvalues, and \emph{c-semi-stable} (resp. \emph{c-semi-unstable}) if it contains only c-stable or c-critical (resp. c-unstable or c-critical) eigenvalues. The same definitions hold replacing the prefix c- with d- if we replace the expression $\RE(\lambda)$ with $\abs{\lambda}-1$.
\end{definition}
\begin{definition}
Let $\mathcal M \in \C^{k,k}$ be given. A subspace $\cV\subset\C^{k}$ is called \emph{$\mathcal{M}$-neutral} if $x^*\mathcal{M} y=0$ for all $x,\,y\in {\cV}$.
\end{definition}
\begin{definition}\label{con-obs}
Let a~pair $(A,B)\in\C^{n,n}\times\C^{n,m}$ be given. Then
\begin{enumerate}[(i)]
   \item $(A,B)$ is called \emph{controllable} if $\rank[\,s I\!-\!A\,,\,B\,]=n$ for all $s\in\C$;
   \item $(A,B)$ is called \emph{stabilizable} if $\rank[\,s I\!-\!A\,,\,B\,]=n$ for all $s\in\overline{\C^+}$.
\end{enumerate}
\end{definition}
\begin{definition}
Given $\gamma \in \mathbb R$, $\gamma \neq 0$, the \emph{Cayley transform} of a regular pencil $s\mathcal E-\mathcal A$ is the pencil
\[
 s\mathcal E_\gamma-\mathcal A_\gamma,\qquad \mathcal E_\gamma=A+\gamma\mathcal E_\gamma, \quad \mathcal A_\gamma=A-\gamma\mathcal E.
\]
\end{definition}
The Cayley transform preserves left and right eigenvectors and Jordan/Kronecker chains, while transforms the associated eigenvalues according to the map $\mathcal C:\lambda \mapsto \frac{\lambda-\gamma}{\lambda+\gamma}$, $\lambda \in \mathbb C \cup \infty$. In particular, Kronecker blocks of size $k$ for $\lambda$ are mapped to Kronecker blocks of size $k$ for $\mathcal C(\lambda)$.

\section{Solvability of Lur'e equations}\label{sec:solv}
In this part we collect theoretical results being equivalent for the solvability of Lur'e equations. For convenience, we will call a Hermitian matrix $X$ a~solution of the Lur'e equations if 
(\ref{eq:lure}) is fulfilled for some $K\in\C^{p,n}$, $L\in\C^{p,m}$.

We now introduce some further concepts which are used to characterize solvability of the Lur'e equations.
\begin{definition}\label{def:spdensfct}
For Lur'e equations (\ref{eq:lure}), the {\em spectral density function} is defined as
\begin{equation}
\Phi(i\omega)=\begin{bmatrix}(i\omega I-A)^{-1}B\\I_m\end{bmatrix}^*\begin{bmatrix}Q&C\\C^*&R\end{bmatrix}
\begin{bmatrix}(i\omega I-A)^{-1}B\\I_m\end{bmatrix}
\label{eq:popfunc}
\end{equation}
\end{definition}

In several works, the spectral density function is also known as {\em Popov function}.
\begin{definition}\label{def:lmi}
For Lur'e equations (\ref{eq:lure}), the {\em associated linear matrix inequality (LMI)} is defined as
\begin{equation}\begin{bmatrix}A^*Y+YA+Q& YB+C\\B^*Y+C^*&R\end{bmatrix}\geq0.\label{eq:lmi}\end{equation}
The solution set of the LMI is defined as 
\begin{equation}\mathcal{S}_{LMI}=\{Y\in\C^{n,n}\,:\,Y\text{ is Hermitian and (\ref{eq:lmi}) holds true}\}.\label{eq:lmi_sol}\end{equation}
The LMI (\ref{eq:lmi}) is called {\em feasible} if $\mathcal{S}_{LMI}\neq\emptyset$.
\end{definition}
It can be readily verified that $Y\in\mathcal{S}_{LMI}$ solves the Lur'e equations if it minimizes the rank of (\ref{eq:lmi}).
We now collect some known equivalent solvability criteria Lur'e equations. In the following we require that the pair $(A,B)$ is stabilizable. Note that this assumption can be further weakened
by reducing it to {\em sign-controllability} \cite{Rei09}. This is not considered here in more detail.
\begin{theorem}\label{thm:lure_solvability}
Let the Lur'e equations (\ref{eq:lure}) with associated even matrix pencil $s\mathcal{E}-\mathcal{A}$ as in (\ref{eq:evmatpen}) and spectral density function $\Phi$ as in (\ref{eq:popfunc}) be given. Assume that at least one of the claims
\begin{enumerate}[(i)]
\item the pair $(A,B)$ is stabilizable and the pencil $s\mathcal{E}-\mathcal{A}$ as in (\ref{eq:evmatpen}) is regular;
\item the pair $(A,B)$ is controllable;
\end{enumerate}
holds true.
Then the following statements are equivalent:
\begin{enumerate}
\item There exists a solution $(X,K,L)$ of the Lur'e equations.
\item The LMI (\ref{eq:lmi}) is feasible
\item For all $\omega\in\R$ with $i\omega\notin\sigma(A)$ holds $\Phi(i\omega)\geq0$;
\item In the EKCF of $s\mathcal{E}-\mathcal{A}$, all blocks of type E2 have positive signature and even size, and all
blocks of type E3 have negative sign and odd size.
\item In the EKCF of $s\mathcal{E}-\mathcal{A}$, all blocks of type E2 have even size, and all
blocks of type E3 have negative sign and odd size.
\end{enumerate}
In particular, solutions of the Lur'e equations fulfill $(X,K,L)\in\C^{n,n}\times\C^{n,p}\times\C^{m,p}$ with $p=\normalrank\Phi$.
\end{theorem}~\\
~\\
It is shown in \cite{Rei09} that $m-\normalrank\Phi$ is to the number of blocks of type E4 in an EKCF of $s\mathcal{E}-\mathcal{A}$. In particular, the pencil $s\mathcal{E}-\mathcal{A}$ is regular if and only if $\Phi$ has full normal rank.

Now we place particular emphasis on the so-called {\em maximal solution}.
\begin{theorem}\label{thm:maxsol}
Let the Lur'e equations (\ref{eq:lure}) be given with stabilizable pair $(A,B)$. Assume that $\mathcal{S}_{LMI}$ as defined in (\ref{eq:lmi_sol}) is non-empty. 
Then there exists a~solution $X_+$ of the Lur'e equations that is maximal in the sense that for all $Y\in\mathcal{S}_{LMI}$ holds
\[Y\leq X_+.\]
\end{theorem}
The following result \cite{Rei09} states that the maximal solution can be constructed via the c-stable deflating subspace of the associated even matrix pencil $s\mathcal{E}-\mathcal{A}$.
\begin{theorem}\label{thm:lagsplur}
Let the Lur'e equations (\ref{eq:lure}) be given with stabilizable pair $(A,B)$. Assume that $\mathcal{S}_{LMI}$ as defined in (\ref{eq:lmi_sol}) is non-empty. Then
\begin{equation}\label{eq:lagrsol}
\begin{bmatrix}
X_+ & 0\\
I_n & 0\\
0 & I_m
\end{bmatrix}
\end{equation}
spans the unique $(n+m)$-dimensional semi-c-stable $\mathcal{E}-neutral$ subspace of the pencil \eqref{eq:evmatpen}.
\end{theorem}
The above theorem states that the maximal solution can be expressed in terms of a special deflating subspace of $s\mathcal{E}-\mathcal{A}$. By means of the EKCF, this space can be constructed from the matrix $U\in\C^{2n+m,2n+m}$ bringing the pencil $s\mathcal{E}-\mathcal{A}$ into even Kronecker form (\ref{eq:kronform_ev}).
Considering the partitioning
$U=[\,U_1\,,\ldots,\,U_k\,]$ according to the block structure of the EKCF, a~matrix $V\in\C^{2n+m,n+m}$ spanning the desired deflating subspace can be constructed by
\begin{equation} V=\begin{bmatrix}V_1&\ldots&V_k\end{bmatrix}\quad\text{ for }V_j=U_jZ_j,\label{eq:lagrparam}\end{equation}
where
\[
Z_j=\begin{cases}
\text{$[\,I_{k_j}\,,\,0_{k_j}\,]^T,$}&\text{if $\mathcal{D}_j$ is of type E1,}\\
\text{$[\,I_{k_j/2}\,,\,0_{k_j/2}\,]^T,$}\quad&\text{if $\mathcal{D}_j$ is of type E2,}\\
\text{$[\,I_{(k_j+1)/2}\,,\,0_{(k_j-1)/2}\,]^T,$}\quad&\text{if $\mathcal{D}_j$ is of type E3,}\\
\text{$[\,I_{k_j}\,,\,0_{k_j+1}\,]^T,$}\quad&\text{if $\mathcal{D}_j$ is of type E4.}
\end{cases}
\]
In particular, the desired subspace contains all the vectors belonging to the Kronecker chains relative to c-stable eigenvalues, no vectors from the Kronecker chains relative to c-unstable eigenvalues, the first $k_j/2$ vectors from the chains relative to c-critical eigenvalues, and the first $(k_j+1)/2$ from the chains relative to eigenvalues at infinity.

\section{Outline of SDA}\label{sec:SDA}
The structured doubling algorithm (SDA) is a matrix iteration which computes two special deflating subspaces of a matrix pencil. It was introduced by Anderson \cite{And78} as an algorithm for the solution of a discrete-time algebraic Riccati equation, and later adapted to many other equations and explained in terms of matrix pencils in several papers by Wen-Wei Lin and others \cite{CFL05,LinXu06,HuaLin09,CCGHLX09}. It is strongly related to the sign function method and to the disc method for matrix pencils \cite{BDG97, Ben97}.

\begin{theorem}[\cite{BDG97}] \label{sdasquaring}
Let $\mathcal A-s\mathcal E$ be a regular matrix pencil with $\mathcal A,\mathcal E\in \mathbb R^{N+M,N+M}$, and let $\mathcal A_*-s\mathcal E_*$ be a regular pencil of the same size with $\mathcal A_* \mathcal E=\mathcal E_*\mathcal A$. Then
\begin{enumerate}
\item the pencil $\mathcal A_*\mathcal A-s\mathcal E_*\mathcal E$ is regular and has the same right deflating subspaces as $\mathcal A-s\mathcal E$
\item its eigenvalues are the square of the eigenvalues of the original pencil.
\end{enumerate}
\end{theorem}
This result is far from surprising in the case in which $\mathcal E$ is invertible: in this case, the eigenvalues and right deflating subspaces of $\mathcal A-s\mathcal E$ correspond to the eigenvalues and right invariant subspaces of $\mathcal E^{-1}\mathcal A$, and it is simple to check that the conditions imposed on $\mathcal A_*,\mathcal E_*$ imply $(\mathcal E_* \mathcal E)^{-1}(\mathcal A_*\mathcal A)=(\mathcal E^{-1} \mathcal A)^2$. Thus the map $(\mathcal A,\mathcal E)\mapsto (\mathcal A_*\mathcal A,\mathcal E_*\mathcal E)$ is a way to extend the concept of squaring to matrix pencils.

A pencil $\mathcal A-s\mathcal E$ with $\mathcal A,\mathcal E\in \mathbb R^{N+M,N+M}$ is said to be in \emph{standard symplectic-like form I (SSF-I)} if it can be written as
\begin{equation}\label{ssf}
\mathcal A=\begin{bmatrix} E & 0\\-H & I_M \end{bmatrix},\quad \mathcal E=\begin{bmatrix} I_N & -G\\0 & F \end{bmatrix},
\end{equation}
where the block sizes are such that $E\in \mathbb R^{N,N}$ and $F\in \mathbb R^{M,M}$. Note that a pencil in SSF-I is always regular.
When a pencil $\mathcal A-s\mathcal E$ is in SSF-I, a choice of $\mathcal A_*-s\mathcal E_*$ satisfying the requirements of Theorem~\ref{sdasquaring} is
\[
\mathcal A_*=\begin{bmatrix} E(I_N-GH)^{-1} & 0\\-F(I_M-HG)^{-1}H & I_N \end{bmatrix},\quad \mathcal E_*=\begin{bmatrix} I_N & -E(I_N-GH)^{-1}G\\0 & F(I_M-HG)^{-1} \end{bmatrix},
\]
yielding a new pencil $\mathcal A_*\mathcal A-s\mathcal E_*\mathcal E$ which is still in SSF-I:
\[
\widetilde{\mathcal A}=\begin{bmatrix} E(I_N-GH)^{-1}E & 0\\-(H+F(I_M-HG)^{-1}HE) & I_N \end{bmatrix},\quad \widetilde{\mathcal E}=\begin{bmatrix} I_N & -(G+E(I_N-GH)^{-1}GF)\\0 & F(I_M-HG)^{-1}F \end{bmatrix}.
\]
The only hypothesis needed here is that $I-GH$ and $I-HG$ are nonsingular. In fact, by the Sherman-Morrison formula, they are either both singular or both nonsingular.

We may design Algorithm~\ref{SDA} by repeating this transformation.
\begin{algorithm}[ht]
\SetKwData{Ez}{\ensuremath{E_0}}
\SetKwData{Fz}{\ensuremath{F_0}}
\SetKwData{Gz}{\ensuremath{G_0}}
\SetKwData{Hz}{\ensuremath{H_0}}

\SetKwData{Ek}{\ensuremath{E_k}}
\SetKwData{Fk}{\ensuremath{F_k}}
\SetKwData{Gk}{\ensuremath{G_k}}
\SetKwData{Hk}{\ensuremath{H_k}}
\SetKwData{Ebar}{\ensuremath{E_*}}
\SetKwData{Fbar}{\ensuremath{F_*}}
\SetKwData{Ekp}{\ensuremath{E_{k+1}}}
\SetKwData{Fkp}{\ensuremath{F_{k+1}}}
\SetKwData{Gkp}{\ensuremath{G_{k+1}}}
\SetKwData{Hkp}{\ensuremath{H_{k+1}}}

\SetKwData{XX}{\ensuremath{H_{\infty}}}
\SetKwData{YY}{\ensuremath{G_{\infty}}}
\SetKwData{kk}{\ensuremath{k}}

\SetKwInOut{Input}{input}
\SetKwInOut{Output}{output}
\Input{\Ez, \Fz, \Gz, \Hz defining a pencil in SSF-I}
\Output{\XX,\YY so that the subspaces in \eqref{sdaspaces} are respectively the canonical semi-d-stable and semi-d-unstable deflating subspaces of the given pencil}
$\kk\longleftarrow 0$\;
\While{a suitable stopping criterion is not satisfied}{
$\Ebar \longleftarrow \Ek(I_N-\Gk\Hk)^{-1}$\;
$\Fbar \longleftarrow \Fk(I_M-\Hk\Gk)^{-1}$\;
$\Gkp \longleftarrow \Gk+\Ebar\Gk\Fk$\;
$\Hkp \longleftarrow \Hk+\Fbar\Hk\Ek$\;
$\Ekp \longleftarrow \Ebar \Ek$\;
$\Fkp \longleftarrow \Fbar \Fk$\;
$\kk\longleftarrow \kk+1$\;
}
\XX=\Hk\;
\YY=\Gk\;
\caption{SDA-I}\label{SDA}
\end{algorithm}
Each step of the algorithm costs $\frac{14}3 (M^3+N^3)+6MN(M+N)$ floating point operations. This reduces to $\frac{64}3N^3$ in the case $M=N$.

The following result is proved in \cite{HuaLin09} for the symplectic case and in \cite{CCGHLX09} for several specific matrix equations, but its proof works without changes for our slightly more general case.

We shall call a pencil \emph{weakly d-split} if there exists an $r$ such that:
\begin{itemize}
\item the lengths of the Kronecker chains relative to d-stable eigenvalues sum up to $N-r$;
\item the lengths of the Kronecker chains relative to d-unstable eigenvalues sum up to $M-r$;
\item the lengths $k_j$ of the Kronecker chains relative to d-critical eigenvalues (which must sum up to $2r$ if the two previous properties hold) are all even.
\end{itemize}
In this case, we define the \emph{canonical} d-semi-stable (resp. d-semi-unstable) subspace as the invariant subspace spanned by all the Kronecker chains relative to d-stable (resp. d-unstable) eigenvalues, plus the first $k_j/2$ vectors from each critical chain.

\begin{theorem} Let the pencil \eqref{ssf} be weakly d-split, and suppose that there are two matrices in the form
\begin{equation}\label{sdaspaces}
\begin{bmatrix}
I_N\\H_{\infty}
\end{bmatrix},\quad
\begin{bmatrix}
G_{\infty}\\I_M
\end{bmatrix},
\end{equation}
spanning respectively the canonical d-stable and d-unstable deflating subspace.
Then for Algorithm~\ref{SDA} it holds that
\begin{itemize}
\item $\norm{E_k}=O(2^{-k})$,
\item $\norm{F_k}=O(2^{-k})$,
\item $\norm{H_{\infty}-H_k}=O(2^{-k})$,
\item $\norm{G_{\infty}-G_k}=O(2^{-k})$.
\end{itemize}
\end{theorem}

Notice that, when $N=M$, a pencil in SSF-I is symplectic if and only if $E^*=F$, $G=G^*$ and $H=H^*$. In this case, all the pencils generated by the successive steps of SDA are symplectic, i.e., at each step $k$ we have $E_k^*=F_k$, $G_k=G_k^*$, $H_k=H_k^*$. The implementation can be slightly simplified, since there is no need to compute $E_{k+1}$ and $F_{k+1}$ separately, nor to invert both $I_N-G_k H_k$ and $I_M-H_K G_K$, as the second matrix of both pairs is the transposed of the first.

\section{A method to compute the SSF-I of a pencil}\label{sec:ssfilemma}
One can transform a regular pencil into SSF-I easily using the following result.
\begin{theorem}\label{thm:bmf}
Let $\mathcal A-s\mathcal E$ be a matrix pencil with $\mathcal A, \mathcal E \in \mathbb R^{N+M,N+M}$, and partition both matrices as 
\[
\mathcal A=\begin{bmatrix} \mathcal A_1 & \mathcal A_2\end{bmatrix}
\,\quad
\mathcal E=\begin{bmatrix} \mathcal E_1 & \mathcal E_2\end{bmatrix}
\]
with $\mathcal A_1,\mathcal E_1\in\mathbb R^{N+M,N}$ and $\mathcal A_2,\mathcal E_2\in \mathbb R^{N+M,M}$. A SSF-I pencil having the same eigenvalues and right deflating subspaces of the original pencil exists if and only if
\begin{equation}\label{toinvert}
\begin{bmatrix}
\mathcal E_1 & \mathcal A_2
\end{bmatrix}
\end{equation}
is nonsingular; in this case, it holds
\begin{equation}\label{bmf}
\begin{bmatrix}
E & -G\\
-H &F
\end{bmatrix}
=
\begin{bmatrix}
\mathcal E_1 & \mathcal A_2
\end{bmatrix}^{-1}
\begin{bmatrix}
\mathcal A_1 & \mathcal E_2
\end{bmatrix}.
\end{equation}
\end{theorem}
\begin{proof}
We are looking for a matrix $Q$ such that
\[
Q
\begin{bmatrix} \mathcal A_1 & \mathcal A_2\end{bmatrix}-s Q\begin{bmatrix} \mathcal E_1 & \mathcal E_2\end{bmatrix}=
\begin{bmatrix}
E & 0\\
-H & I
\end{bmatrix}
-s
\begin{bmatrix}
I & -G\\0 & F
\end{bmatrix}.
\]
By taking only some of the blocks from the above equation, we get
\[
Q\mathcal E_1=\begin{bmatrix}I\\0\end{bmatrix},\quad
Q\mathcal A_2=\begin{bmatrix}0\\I\end{bmatrix},
\]
i.e.,
\[
Q \begin{bmatrix}
\mathcal E_1 & \mathcal A_2
\end{bmatrix}
=
\begin{bmatrix}
I & 0 \\ 0 & I
\end{bmatrix}
,
\]
thus $Q$ must be the inverse of the matrix in \eqref{toinvert}.

On the other hand, taking the other two blocks we get
\[
Q\mathcal A_1=\begin{bmatrix}E\\-H\end{bmatrix},\quad
Q\mathcal E_2=\begin{bmatrix}-G\\F\end{bmatrix},
\]
which promptly yields \eqref{bmf}.
\end{proof}
This formula is strictly related to the principal pivot transform (PPT) \cite{Tsa00}.

We point out an interesting application of Theorem~\ref{thm:bmf}, which is not related to the rest of the paper. SDA is often used in the solution of nonsymmetric algebraic Riccati equations (NARE) \cite{GuoLaub99}, where it is applied to the Cayley transform of the matrix
\[
\mathcal H=\begin{bmatrix}
D & -C\\
B & -A
\end{bmatrix},
\]
with $A$, $B$, $C$, $D$ blocks of suitable size associated with the coefficients of the problem. In pencil form, its Cayley transform given by $(\mathcal H-\gamma I)-s(\mathcal H+\gamma I)$, thus \eqref{bmf} becomes
\begin{equation}\label{bmf-nare}
\begin{bmatrix}
E & -G\\
-H &F
\end{bmatrix}
=
\begin{bmatrix}
D+\gamma I & -C\\
B & -A-\gamma I
\end{bmatrix}^{-1}
\begin{bmatrix}
D-\gamma I & -C\\
B & -A+\gamma I
\end{bmatrix}.
\end{equation}
% 32/3N^3 vs. 68/3N^3 
This formula is more compact to write and more computationally effective than the one suggested by Guo \emph{et al.} \cite{GLX06}. In fact, their expressions for the starting blocks involve computing the inverses of two $N\times N$ and $M\times M$ matrices, which are indeed the $(1,1)$ and $(2,2)$ blocks of the matrix to be inverted in \eqref{bmf-nare} and their Schur complements. Clearly, two these inversions are redundant in \eqref{bmf-nare}, which requires more or less half of the computational cost with respect to the original formulas in \cite{GLX06}.

\section{A reduced Lur'e pencil}\label{sec:reduction}
Let $s\mathcal E - \mathcal A$ be the pencil in \eqref{eq:evmatpen}. By Theorem~\ref{thm:bmf}, the SSF-I form of its Cayley transform (assuming $M=n+m$, $N=n$) is given by
\begin{equation}\label{bmf-lure}
\begin{bmatrix}E & -G\\-H & F \end{bmatrix}=
\begin{bmatrix}0 & A-\gamma I & B\\A^*-\gamma I & Q & C\\B^* &C^*& R\end{bmatrix}^{-1}
\begin{bmatrix}0 & A+\gamma I & B\\A^*+\gamma I & Q & C\\B^* &C^*& R\end{bmatrix}.
\end{equation}
Let now
\[
\widetilde A:=\begin{bmatrix}
0 & A-\gamma I\\
A^*-\gamma I & Q
\end{bmatrix},
S=\begin{bmatrix}
0 & I_n\\
I_n & I
\end{bmatrix},
\]
and assume that both $\widetilde A$ and its Schur complement
\[
R-\begin{bmatrix}B^* & C^*\end{bmatrix}\widetilde A^{-1}\begin{bmatrix}B\\C\end{bmatrix}=\Phi(\gamma)
\]
(here $\Phi$ is the same function as in \eqref{eq:popfunc}, as one can verify by expanding both expressions) are nonsingular. In this case we can perform the inversion with the help of a block LDU factorization. Tedious computations lead to a matrix of the form
\[
\begin{bmatrix}
\widehat A & 0\\
\widehat B & I_m
\end{bmatrix},
\]
with
\begin{equation}\label{tedious}
\widehat A=I+2\gamma \widetilde A^{-1} S + 2\gamma \widetilde A^{-1} \begin{bmatrix}B\\C\end{bmatrix} \Phi(\gamma)^{-1}   \begin{bmatrix}B^* & C^*\end{bmatrix}\widetilde A^{-1} S, \quad S=\begin{bmatrix}0 & I\\ I & 0\end{bmatrix}
\end{equation}
In the blocks used in SSF-I, this means that
\[
F=
\begin{bmatrix}
\widehat F & 0\\
\ast & I_m
\end{bmatrix}
,\quad
G=
\begin{bmatrix}
\widehat G & 0
\end{bmatrix}
,\quad
H=
\begin{bmatrix}
\widehat H\\ \ast
\end{bmatrix},
\]
where the blocks $\widehat F$, $\widehat G$, $\widehat H$ have size $n\times n$. It follows that a special right deflating subspace of this pencil is
\[
\begin{bmatrix}
0_{2n \times m}\\ I_m
\end{bmatrix},
\]
whose only eigenvalue is 1 with algebraic and geometric multiplicity $m$, while the deflating subspaces relative to the other eigenvalues are in the form
\[
 \begin{bmatrix}
  V\\
  \ast
 \end{bmatrix},
\]
where $V$ has $2n$ rows and is a deflating subspace of the reduced pencil
\begin{equation}\label{reduced-pencil}
s\begin{bmatrix} I_n & - \widehat G\\0 & \widehat F \end{bmatrix}-
\begin{bmatrix} E & 0\\-\widehat H & I_n \end{bmatrix}.
\end{equation}
Using \eqref{tedious} and the fact that $\widetilde A$ and $\Phi(\gamma)$ are symmetric, one sees that $\widehat AS$ is symmetric, too. This means that $E^*=F$ and $G=G^*$, $H=H^*$, that is, the pencil \eqref{reduced-pencil} is symplectic. 

The pencil \eqref{reduced-pencil} is given by $P^*(s\mathcal E-\mathcal A)P$, where $P$ is the projection on
\[
\left(\spa\left(\begin{bmatrix}0\\0\\I_m\end{bmatrix}\right)\right)^{\perp}=(\ker \mathcal E)^{\perp}.
\]
With this characterization, it is easy to derive the KCF of \eqref{reduced-pencil} from that of the Cayley transform of \eqref{eq:evmatpen}. We see that $\ker \mathcal E$ is the space spanned by the first column of each W2 block (as a corollary, we see that there are exactly $m=\dim\ker E$ such blocks). These blocks are transformed into blocks W1 with $\lambda=1$ by the Cayley transform. Thus projecting on their orthogonal complement corresponds to dropping the first row and column from each of the W1 blocks relative to $\lambda=1$. In particular, it follows that if the criteria in Theorem~\ref{thm:lure_solvability} hold, then every Kronecker block of \eqref{reduced-pencil} relative to an eigenvalue $\lambda$ on the unit circle has even size. Therefore, the reduced pencil \eqref{reduced-pencil} is weakly d-split. By considering which vectors are needed from each Kronecker chain to form the subspace in \eqref{eq:lagrparam}, we get therefore the following result.
\begin{theorem}
Let $\mathcal V$ be the unique $(n+m)$-dimensional c-semi-stable $\mathcal E$-neutral deflating subspace of \eqref{eq:evmatpen}. Then, there is a matrix $V_2\in\mathcal C^{n,m}$ such that
\[
\mathcal V=\spa\begin{bmatrix}V_1 & 0\\ V_2 & I_m \end{bmatrix},
\]
where $V_1$ spans the canonical d-semi-unstable subspace of the pencil \eqref{reduced-pencil}. Moreover, if $\spa(V_1)$ admits a basis in the form
\[
\begin{bmatrix}
X_+\\
I_n
\end{bmatrix},
\]
then $X_+$ is the maximal solution of the Lur'e equation \eqref{eq:lure}.
\end{theorem}
In other words, $X$ is the canonical weakly stabilizing solution of the DARE
\begin{equation}\label{newDARE}
 X=EX(I-\widehat HX)^{-1}E^*+\widehat G.
\end{equation}

\section{Implementation of SDA}\label{sec:details}
Based on the results of the previous sections, we can use the SDA-I algorithm to compute the solution to a Lur'e equation. The resulting algorithm is reported as Algorithm~\ref{SDA-l}.

As we saw in Section~\ref{sec:SDA}, the symplecticity of the pencil is preserved during the SDA iterations, and helps reducing the computational cost of the iteration. Moreover, this has the additional feature of preserving the eigenvalue symmetry of the original pencil along the iteration.

The explicit computation of (a possible choice of) $K$ and $L$ are typically not needed in the applications of the Lur'e equations. If they are needed, they can be computed using the fact that
\begin{equation}\label{eq:fullrank}
\begin{bmatrix}
A^* X + X^* A + Q & XB+C\\B^*X^*+C^* R
\end{bmatrix}
=
\begin{bmatrix}
K^*\\L^*
\end{bmatrix}
\begin{bmatrix}
K & L
\end{bmatrix}
\end{equation}
is a full rank decomposition.

\begin{algorithm}[ht]
\caption{A SDA algorithm for the maximal solution of a Lur'e equation}\label{SDA-l}
\SetKwData{AAA}{\ensuremath{A}}
\SetKwData{BBB}{\ensuremath{B}}
\SetKwData{CCC}{\ensuremath{C}}
\SetKwData{QQQ}{\ensuremath{Q}}
\SetKwData{RRR}{\ensuremath{R}}
\SetKwData{XXX}{\ensuremath{X}}
\SetKwData{KKK}{\ensuremath{K}}
\SetKwData{LLL}{\ensuremath{L}}
\SetKwData{TTT}{\ensuremath{T}}
\SetKwData{EEE}{\ensuremath{E}}
\SetKwData{FFF}{\ensuremath{F}}
\SetKwData{GGG}{\ensuremath{G}}
\SetKwData{HHH}{\ensuremath{H}}
\SetKwData{Ginf}{\ensuremath{G_\infty}}
\SetKwData{Hinf}{\ensuremath{H_\infty}}

\SetKwInOut{Input}{input}
\SetKwInOut{Output}{output}
\SetKwFunction{SDAI}{SDA-I}
\Input{\AAA, \BBB, \CCC, \QQQ, \RRR defining a Lur'e equation} 
\Output{The weakly stabilizing solution \XXX (and optionally \KKK and \LLL)}
Choose a suitable $\gamma>0$\;
Compute
\[
\TTT\longleftarrow
\begin{bmatrix}
0 & \AAA-\gamma I & \BBB\\
\AAA^*-\gamma I & \QQQ & \CCC\\
\BBB^* & \CCC^*& \RRR\\
\end{bmatrix}^{-1}
\begin{bmatrix}
0 & \AAA+\gamma I\\
\AAA^*+\gamma I & \QQQ\\
\BBB^* & \CCC^*\\
\end{bmatrix};
\]
Partition
\[
T=\begin{bmatrix}
\EEE & -\GGG\\
-\HHH & {\EEE}^*\\
\ast & \ast
\end{bmatrix}\;
\]
Use \SDAI on \EEE, $\FFF=\EEE^*$, \GGG,\HHH to compute $\Ginf$, $\Hinf$, and set $X=\Ginf$\;
\If{\KKK and \LLL are needed}{
 Compute $\Sigma$ and $V^*$ corresponding to the first $m$ singular vectors of the SVD of \eqref{eq:fullrank}\;
Set $\begin{bmatrix}K & L\end{bmatrix}\longleftarrow \Sigma^{1/2}V^*$\;
}
\end{algorithm}

The accuracy of the computed solution depends also on an appropriate choice of $\gamma$. Clearly, two goals have to be considered in the choice:
\begin{enumerate}
\item the matrix to invert in \eqref{bmf-lure} should be well-conditioned;
\item the condition number of the resulting problem, i.e, the separation between the stable and unstable subspace of the Cayley-transformed pencil \eqref{reduced-pencil}, should not be too small.
\end{enumerate}
While the impact of the first factor is easy to measure, the second one poses a more significant problem, since there are no \emph{a priori} estimates for the conditioning of a subspace separation problem. Nevertheless, we may try to understand how the choice of the parameter $\gamma$ of the Cayley transform affects the conditioning. Roughly speaking, the conditioning of the invariant subspace depends on the distance between its eigenvalues and those of the complementary subspace \cite{GVL}. The eigenvalues of the transformed pencil are given by $\frac{\lambda-\gamma}{\lambda+\gamma}=1-\frac{2\gamma}{\lambda+\gamma}$, thus they tend to cluster around 1 for small values of $\gamma$, which is undesirable. The closest to 1 is the one for which $\lambda+\gamma$ has the largest modulus; we may take as a crude approximation $\rho(A)+\gamma$, which can be further approximated loosely with $\norm{A}_1+\gamma$. Therefore, as a heuristic to choose a reasonable value of $\gamma$, we may look for a small value of
\[
f(\gamma)=\max\left(\condest\left(\begin{bmatrix}\mathcal E_1 & \mathcal A_2\end{bmatrix}\right), \frac{\norm{A}_1+\gamma}{2\gamma} \right),
\]
where $\condest(\cdot)$ is the condition number estimate given by Matlab. Following the strategy of \cite{CFL05}, we chose to make five steps of the golden section search method \cite{LuYe08} on $f(\gamma)$ in order to get a reasonably good value of the objective function without devoting too much time to this ancillary computation.

However, we point out that in \cite{CFL05}, a different $f(\gamma)$ is used, which apparently only takes into account the first of our two goals. The choice of the function is based on an error analysis of their formulas for the starting blocks of SDA. Due to Theorem~\ref{bmf}, this part of the error analysis can be simplified to checking $\condest\left(\begin{bmatrix}\mathcal E_1 & \mathcal A_2\end{bmatrix}\right)$.

\section{Numerical experiments}\label{sec:experiments}
We have implemented Algorithm~\ref{SDA-l} (SDA-L) using Matlab, and tested it on the following test problems.
\begin{description}
\item[P1] a Lur'e equation with a random stable matrix $A\in \mathbb R^{n,n}$, a random $C=B$, $Q=0$ and $R$ the $m\times m$ matrix with all the entries equal to 1, with $\rank(R)=1$. Namely, $B$ was generated with the Matlab command \lstinline{B=rand(n,m)}, while to ensure a stable $A$ a we used the following more complex sequence of commands: \lstinline{V=randn(n);} \lstinline{W=randn(n);} \lstinline{A=-V*V'-W+W';}
\item[P2] a set of problems motivated from real-world examples, taken with some modifications from the benchmark set \lstinline{carex} \cite{carex}. Namely, we took Examples 3 to 6 (the real-world applicative problems) of this paper, which are a set of real-world problems varying in size and numerical characteristics, and changed the value of $R$ to get a singular problem. In the original versions of all examples, $R$ is the identity matrix of appropriate size; we simply replaced its $(1,1)$ entry with 0, in order to get a singular problem.
\item[P3] a highly ill-conditioned high-index problem with $m=1$, $A=I_n+N_n$, $B=e_n$ (the last vector of the canonical basis for $\mathbb R^n$), $C=-B$, $R=0$, $Q=-\operatorname{tridiag}(1,2,1)$. Such a problem corresponds to a Kronecker chain of length $2n+1$ associated to an infinite eigenvalue, and its canonical semi-stable solution is $X=I$. Notice that the conditioning of the invariant subspace problem in this case is $\epsilon^{1/(2n+1)}$, for an unstructured perturbation of the input data of the order of the machine precision $\epsilon$ \cite[section 16.5]{GLR}.
\end{description}

The results of SDA-L are compared to those of a regularization method as the one described in \eqref{eq:are_pert}, for different values of the regularization parameter $\varepsilon$. After the regularization, the equations are solved using Algorithm~\ref{SDA} after a Cayley transform with the same parameter $\gamma$ (R+S), or with the matrix sign method with determinant scaling \cite{MehBook,HigBook} (R+N). We point out that the control toolbox of Matlab contains a command \lstinline{gcare} that solves a so-called generalized continuous-time algebraic Riccati equation based on a pencil in a form apparently equivalent to that in \eqref{eq:evmatpen}. In fact, this command is not designed to deal with a singular $R$, nor with eigenvalues numerically on the imaginary axis. Therefore, when applied to nearly all the following experiments, this command fails reporting the presence of eigenvalues too close to the imaginary axis.

For the problem P3, where an analytical solution $X=I$ is known, we reported the values of the forward error
\[
\frac{\norm{\tilde X-X}_F}{\norm{X}_F}.
\] 
For P1 and P2, for which no analytical solution is available, we computed instead the relative residual of the Lur'e equations in matrix form
\[
\frac{\norm{
\begin{bmatrix}
A'X+XA+Q & XB+C\\ B^* X^* + C^* & R
\end{bmatrix}
-
\begin{bmatrix}K^*\\ L^*\end{bmatrix}
\begin{bmatrix}K & L\end{bmatrix}
}_F}{\norm{\begin{bmatrix}
A'X+XA+Q & XB+C\\ B^* X^* + C^* & R
\end{bmatrix}
}_F}
\]
(see \eqref{eq:fullrank}).
A star $\star$ in the data denotes convergence failure.
\begin{figure}
\caption{Relative residual for P1}
\begin{tabular}{cc|c|ccc|c}
$n$ & $m$ & SDA-L & R+S $\varepsilon=10^{-6}$ & R+S $\varepsilon=10^{-8}$ & R+S $\varepsilon=10^{-12}$ & R+N $\varepsilon=10^{-8}$\\
\hline
10 & 3 & 1E-15 & 2E-8 & 4E-10 & 3E-6 & 3E-10\\
50 & 5 & 3E-14 & 8E-9 & 1E-7 & 2E-1 & 4E-10\\
500 & 10 & 7E-14 & 2E-9 & 1E-7 & 1E-1 & $\star$\\
\end{tabular}
\end{figure}
\begin{figure}
\caption{Relative residual for P2}
\begin{tabular}{c|c|ccc|c}
Problem \# & SDA-L & R+S $\varepsilon=10^{-6}$ & R+S $\varepsilon=10^{-8}$ & R+S $\varepsilon=10^{-12}$ & R+N $\varepsilon=10^{-8}$\\
\hline
3 & 6E-16 & 1E-7 & 9E-10 & 8E-6 & 1E-9\\
4 & 9E-16 & 6E-7 & 6E-9 & 2E-7 & 6E-9\\
5 & 6E-15 & 3E-7 & 1E-9 & 3E-8 & 1E-9\\
6 & 2E-15 & 6E-12 & 2E-12 & 1E-11 & 4E-13\\
\end{tabular}
\end{figure}

\begin{figure}
\caption{Forward error for P3}
\begin{tabular}{c|c|ccc|c}
n & SDA-L & R+S $\varepsilon=10^{-6}$ & R+S $\varepsilon=10^{-8}$ & R+S $\varepsilon=10^{-12}$ & R+N $\varepsilon=10^{-8}$\\
\hline
1 & 1E-8 & 1E-3 & 1E-4 & 1E-6 & 1E-4 \\
2 & 5E-5 & 3E-2 & 1E-2 & 3E-2 & $\star$\\
3 & 2E-3 & 1E-1 & 5E-2 & 2E+0 & $\star$\\
4 & 1E-2 & 4E-1 & 1E-1 & 8E-1 & $\star$\\
5 & 6E-2 & 1E+0 & 4E-1 & 2E+0 & $\star$
\end{tabular}
\end{figure}

We see that in all the experiments our solution method obtains a better result than the ones based on regularization. 
\section{Conclusion and open issues}\label{sec:conc}
In this work we have introduced a~new numerical method for the solution of Lur'e matrix equations. Unlike previous methods based on regularization, this approach allows one to solve the original equation without introducing any artificious perturbation.

The first step of this approach is applying a Cayley transform to convert the problem to an equivalent discrete-time pencil. In this new form, the infinite eigenvalues can be easily deflated, reducing the problem to a discrete-time algebraic Riccati equation with eigenvalues on the unit circle. For the solution of this latter equation, the structured-preserving doubling algorithm was chosen, due to its good behaviour in presence of eigenvalues on the unit circle, as proved by the convergence results in \cite{HuaLin09}. Direct methods, such as the symplectic eigensolvers presented in \cite{FasBook}, could also be used for the solution of the deflated DARE.

The numerical experiments confirm the effectiveness of our new approach for regular matrix pencils. It is not clear whether the same method can be adapted to work in cases in which the pencil \eqref{eq:evmatpen} is singular, which may indeed happen in the contest of Lur'e equations. Another issue is finding a method to exploit the low-rank structure of $Q$ (when present). These further developments are currently under our investigation.

\bigskip\bigskip
\bibliographystyle{siam}
\bibliography{lure}

\begin{thebibliography}{10}

\bibitem{Ande66}
{\sc B.D.O. Anderson}, {\em Algebraic description of bounded real matrices},
  Electronics Letters, 2 (1966), pp.~464--465.

\bibitem{AndNew68}
{\sc B.D.O. Anderson and R.W. Newcomb}, {\em Impedance synthesis via
  state-space techniques}, Proc. IEE, 115 (1968), pp.~928--936.

\bibitem{AndeV73}
{\sc B.D.O. Anderson and S.~Vongpanitlerd}, {\em Network Analysis and
  Synthesis}, Prentice Hall, Englewood Cliffs, NJ, 1973.

\bibitem{And78}
{\sc Brian D.~O. Anderson}, {\em Second-order convergent algorithms for the
  steady-state {R}iccati equation}, Internat. J. Control, 28 (1978),
  pp.~295--306.

\bibitem{BDG97}
{\sc Zhaojun Bai, James Demmel, and Ming Gu}, {\em An inverse free parallel
  spectral divide and conquer algorithm for nonsymmetric eigenproblems}, Numer.
  Math., 76 (1997), pp.~279--308.

\bibitem{Bar07a}
{\sc N.E. Barabanov}, {\em Kalman-{Y}akubovich lemma in general
  finite-dimensional case}, Int. J. Robust Nonlinear Control, 17 (2007),
  pp.~369--386.

\bibitem{Ben97}
{\sc P.~Benner}, {\em Contributions to the Numerical Solution of Algebraic
  Riccati Equations and Related Eigenvalue Problems}, doctoral dissertation,
  Fakult\"at für Mathematik, TU Chemnitz-Zwickau, Chemnitz, 1997.
\newblock Published by Logos-Verlag, Berlin.

\bibitem{carex}
{\sc Peter Benner, Alan Laub, and Volker Mehrmann}, {\em A collection of
  benchmark examples for the numerical solution of algebraic {R}iccati
  equations {I}: the continuous-time case}, Tech. Report SPC 95-22,
  Forschergruppe `Scientific Parallel Computing', Fakult\"at f\"ur Mathematik,
  TU Chemnitz-Zwickau, 1995.
\newblock Version dated February 28, 1996.

\bibitem{ChenWen95}
{\sc X.~Chen and J.T. Wen}, {\em Positive realness preserving model reduction
  with $\mathcal{H}_\infty$ norm error bounds}, Systems Control Lett., 54
  (2005), pp.~361--374.

\bibitem{CCGHLX09}
{\sc Chun-Yueh Chiang, Eric King-Wah Chu, Chun-Hua Guo, Tsung-Ming Huang,
  Wen-Wei Lin, and Shu-Fang Xu}, {\em Convergence analysis of the doubling
  algorithm for several nonlinear matrix equations in the critical case}, SIAM
  J. Matrix Anal. Appl., 31 (2009), pp.~227--247.

\bibitem{CFL05}
{\sc E.~K.-W. Chu, H.-Y. Fan, and W.-W. Lin}, {\em A structure-preserving
  doubling algorithm for continuous-time algebraic {R}iccati equations}, Linear
  Algebra Appl., 396 (2005), pp.~55--80.

\bibitem{CleAnd77a}
{\sc D.J. Clements and B.D.O. Anderson}, {\em Matrix inequality solution to
  linear-quadratic singular control problems}, IEEE Trans. Automatic Control,
  AC-22 (1977), pp.~55--57.

\bibitem{CleAnd77b}
\leavevmode\vrule height 2pt depth -1.6pt width 23pt, {\em Transformational
  solution of singular linear-quadratic control problems}, IEEE Trans.
  Automatic Control, AC-22 (1977), pp.~57--60.

\bibitem{CleAnd78}
\leavevmode\vrule height 2pt depth -1.6pt width 23pt, {\em Singular optimal
  control: the linear-quadratic problem}, vol.~5 of Lecture Notes in Control
  and Information Sciences, Springer-Verlag, Berlin-New York, 1978.

\bibitem{CALM97}
{\sc D.J. Clements, B.D.O. Anderson, A.J. Laub, and L.B. Matson}, {\em Spectral
  factorization with imaginary-axis zeros}, Linear Algebra Appl., 250 (1997),
  pp.~225--252.

\bibitem{CleGlo89}
{\sc D.J. Clements and K.~Glover}, {\em Spectral factorization via {H}ermitian
  pencils}, Linear Algebra Appl., 122-124 (1989), pp.~797--846.

\bibitem{FasBook}
{\sc Heike Fassbender}, {\em Symplectic methods for the symplectic
  eigenproblem}, Kluwer Academic/Plenum Publishers, New York, 2000.

\bibitem{Gant59}
{\sc F.R. Gantmacher}, {\em Theory of Matrices}, Chelsea Publishing Company,
  New York, 1959.

\bibitem{GLR}
{\sc Israel Gohberg, Peter Lancaster, and Leiba Rodman}, {\em Invariant
  subspaces of matrices with applications}, vol.~51 of Classics in Applied
  Mathematics, Society for Industrial and Applied Mathematics (SIAM),
  Philadelphia, PA, 2006.
\newblock Reprint of the 1986 original.

\bibitem{GVL}
{\sc Gene~H. Golub and Charles~F. Van~Loan}, {\em Matrix computations}, Johns
  Hopkins Studies in the Mathematical Sciences, Johns Hopkins University Press,
  Baltimore, MD, third~ed., 1996.

\bibitem{GugeA04}
{\sc S.~Gugercin and A.C. Antoulas}, {\em A survey of model reduction by
  balanced truncation and some new results}, Internat. J. Control, 77 (2004),
  pp.~748--766.

\bibitem{GuoLaub99}
{\sc Chun-Hua Guo and Alan~J. Laub}, {\em On a {N}ewton-like method for solving
  algebraic {R}iccati equations}, SIAM J. Matrix Anal. Appl., 21 (1999),
  pp.~694--698 (electronic).

\bibitem{GLX06}
{\sc Xiao-Xia Guo, Wen-Wei Lin, and Shu-Fang Xu}, {\em A structure-preserving
  doubling algorithm for nonsymmetric algebraic {R}iccati equation}, Numer.
  Math., 103 (2006), pp.~393--412.

\bibitem{HigBook}
{\sc Nicholas~J. Higham}, {\em Functions of matrices}, Society for Industrial
  and Applied Mathematics (SIAM), Philadelphia, PA, 2008.
\newblock Theory and computation.

\bibitem{HuaLin09}
{\sc Tsung-Ming Huang and Wen-Wei Lin}, {\em Structured doubling algorithms for
  weakly stabilizing {H}ermitian solutions of algebraic {R}iccati equations},
  Linear Algebra Appl., 430 (2009), pp.~1452--1478.

\bibitem{JacSp71b}
{\sc D.H. Jacobson and J.L. Speyer}, {\em Necessary and sufficient conditions
  for optimality for singular control problems; a transformation approach}, J.
  Math. Anal. Appl., 33 (1971), pp.~163--187.

\bibitem{JacSp71a}
\leavevmode\vrule height 2pt depth -1.6pt width 23pt, {\em Necessary and
  sufficient conditions for singular control problems: a limit approach}, J.
  Math. Anal. Appl., 34 (1971), pp.~239--266.

\bibitem{Lanc95}
{\sc P.~Lancaster and L.~Rodman}, {\em Algebraic Riccati equations}, Clarendon
  Press, Oxford, 1995.

\bibitem{LinXu06}
{\sc Wen-Wei Lin and Shu-Fang Xu}, {\em Convergence analysis of
  structure-preserving doubling algorithms for {R}iccati-type matrix
  equations}, SIAM J. Matrix Anal. Appl., 28 (2006), pp.~26--39 (electronic).

\bibitem{LuYe08}
{\sc David~G. Luenberger and Yinyu Ye}, {\em Linear and nonlinear programming},
  International Series in Operations Research \& Management Science, 116,
  Springer, New York, third~ed., 2008.

\bibitem{Lur51}
{\sc A.I. Lur'e}, {\em Certain Nonlinear Problems in the Theory of Automatic
  Control}, Gostekhizat, Moscow, Leningrad, 1951.
\newblock Translated into English, H.M. Stationery, 1957.

\bibitem{MehBook}
{\sc V.~L. Mehrmann}, {\em The autonomous linear quadratic control problem},
  vol.~163 of Lecture Notes in Control and Information Sciences,
  Springer-Verlag, Berlin, 1991.
\newblock Theory and numerical solution.

\bibitem{OpdJon88}
{\sc P.C. Opdenacker and E.A. Jonckheere}, {\em A contraction mapping
  preserving balanced reduction scheme and its infinity norm error bounds},
  IEEE Trans. Circuits Syst. I Regul. Pap., 35 (1988), pp.~184--189.

\bibitem{PhilDS03}
{\sc J.R. Phillips, L.~Daniel, and L.M. Silveira}, {\em Guaranteed passive
  balancing transformations for model order reduction}, IEEE Trans.
  Computer-Aided Design Integr. Circuits Syst., 22 (2003), pp.~1027--1041.

\bibitem{Rei09}
{\sc T.~Reis}, {\em Lur'e equations and even matrix pencils}, Technical Report
  09-672, DFG Research Center {\sc Matheon}, Berlin, 2009.
\newblock submitted for publication.

\bibitem{ReiSty08b}
{\sc T.~Reis and T.~Stykel}, {\em Balanced truncation for electrical circuits},
  Preprint 2008/32, Institut f\"ur Mathematik, TU Berlin, 2008.
\newblock submitted for publication.

\bibitem{ReiSty08}
\leavevmode\vrule height 2pt depth -1.6pt width 23pt, {\em Passive and bounded
  real balancing for model reduction of descriptor systems}, Preprint 2008/25,
  Institut f\"ur Mathematik, TU Berlin, 2008.
\newblock submitted for publication.

\bibitem{Tho76}
{\sc R.C. Thompson}, {\em The characteristic polynomial of a principal
  subpencil of a {H}ermitian matrix pencil}, Linear Algebra Appl., 14 (1976),
  pp.~135--177.

\bibitem{Tre87}
{\sc H.L. Trentelman}, {\em Families of linear-quadratic problems: Continuity
  properties}, IEEE Trans. Automatic Control, AC-32 (1987), pp.~323--329.

\bibitem{Tsa00}
{\sc Michael~J. Tsatsomeros}, {\em Principal pivot transforms: properties and
  applications}, Linear Algebra Appl., 307 (2000), pp.~151--165.

\bibitem{WeWaSp94}
{\sc H.~Weiss, Q.~Wang, and J.L. Speyer}, {\em System characterization of
  positive real conditions}, IEEE Trans. Automat. Control, 39 (1994),
  pp.~540--544.

\bibitem{Wil71}
{\sc J.C. Willems}, {\em Least squares stationary optimal control and the
  algebraic {R}iccati equation}, IEEE Trans. Automat. Control, 16 (1971),
  pp.~621--634.

\bibitem{Wil72b}
\leavevmode\vrule height 2pt depth -1.6pt width 23pt, {\em Dissipative
  dynamical systems. {II}: Linear systems with quadratic supply rates}, Arch.
  Ration. Mech. Anal., 45 (1972), pp.~352--393.

\bibitem{Yak85}
{\sc V.A. Yakubovich}, {\em Singular problem in optimal control of linear
  stationary system with quadratic functional}, Siberian Math. J., 26 (1985),
  pp.~148--158.

\bibitem{ZhoDoyGlo96}
{\sc K.~Zhou, J.~Doyle, and K.~Glover}, {\em Robust and Optimal Control},
  Prentice-Hall, Princeton, 1996.

\end{thebibliography}
\end{document}